\documentclass[final]{siamltex}

\usepackage{amsmath}
\usepackage{graphicx}
\usepackage{hyperref}
\usepackage{amsfonts}
\usepackage{epsf,epsfig,psfrag}
\usepackage{setspace}
\usepackage[latin1]{inputenc}
\usepackage[T1]{fontenc}
\usepackage[english]{babel}
\usepackage{times,amssymb,amsmath}

\title{Finitely Unstable Theories and Computational Complexity}
\author{Tuomo Kauranne\thanks{Department of Mathematics and Physics, Lappeenranta University of Technology,
PL 20, 53851 Lappeenranta, Finland ({\tt tuomo.kauranne@lut.fi}).}}

\begin{document}

\maketitle







\begin{abstract} The complexity class $NP$ can be
logically characterized both through existential second order logic $SO\exists$,
as proven by Fagin, and through simulating a Turing machine via the
satisfiability problem of propositional logic SAT, as proven by
Cook. Both theorems involve encoding a Turing machine by a formula in the
corresponding logic and stating that a model of this formula exists if and only if
the Turing machine halts, i.e. the formula is satisfiable iff the 
Turing machine accepts its input. Trakhtenbrot's theorem does the same in first order
logic $FO$. Such different orders of encoding are possible because the
set of all possible configurations of any Turing machine up to any given 
finite time instant can be defined by a finite set of propositional variables,
or is locally represented by a model of fixed finite size. 
In the current paper, we first encode such time-limited computations of a deterministic Turing machine (DTM)
in first order logic. We then take a closer look at DTMs that solve SAT. When the length of the input
string to such a DTM that contains effectively encoded instances of SAT is parameterized by the natural number $M$,
we proceed to show that the corresponding $FO$ theory $SAT_M$ 
has a lower bound on the size of its models that grows almost exponentially with $M$. This lower
bound on model size also translates into a lower bound on the deterministic
time complexity of SAT. 
 
\end{abstract}

\section{Introduction}

Logical definition and computational complexity are intimately
intertwined, as a wide literature witnesses since the seminal theorem by
Fagin \cite{Fagin}. \textit{Fagin's Theorem} states that
\begin{equation}
SO \exists  = NP
\end{equation}
that is, the class of functions computable in polynomial time by
a non-deterministic Turing machine is identical to the class of
sets definable by formulas in second order existential logic. Many
other results in the same spirit of descriptive complexity or finite
model theory have
since been proven, as summarized in, for example, \cite{Baldwin}, \cite{Borger},
\cite{Immerman}, \cite{Flum}, \cite{Libkin}.

At the same time we have \textit{Cook's Theorem} \cite{Cook}
that states that the satisfiability problem of mere "zeroth
order" propositional logic is complete for NP. Cook's Theorem
states that for any halting computation of a given Turing machine,
deterministic or non-deterministic, there exists a propositional
formula - indeed even one in conjunctive normal form - that is
satisfiable if and only if the corresponding truth assignment to
its free variables defines a halting computation of the given
Turing machine. Moreover, this propositional formula is of length
less than $\mathcal{O} (t^3)$ in the number of time steps $t$ taken by
this computation \cite{Salomaa}.

Let us present
propositional formulas in Conjunctive Normal Form, using the encoding
introduced by Immerman in \cite{Immerman}, p. 114. In this encoding, an instance of the
propositional satisfiability problem SAT is defined by two relations over a pair of variables, namely

\begin{tabular}{lll}
 $\bar{P}(w,z)$ & iff & the clause $w$ contains the positive literal $z$ \\
 $\bar{N}(w,z)$ & iff & the clause $w$ contains the negative literal $\lnot z$
\end{tabular}

Let us assume propositional target formulas always to be given
in this form. For the purpose of analysis, however, 
let us extend this encoding by furnishing the binary relations above with a third
variable $\phi$ that refers to a finite propositional formula presented in the CNF form above. 
The domain of each of the three variables is at most countable in size.
Let us now replace the above binary relations with the two
ternary relations $P(w,z,\phi)$ and $N(w,z,\phi)$ and, following \cite{Immerman},
amend them also with the unary relation $\bar{E}(z)$,
with the following intended interpretation

\begin{tabular}{lll}
 $P(w,z,\phi)$ & iff & the clause $w$ contains the positive literal $z$ in the CNF encoding\\
&& of the propositional formula $\phi$\\
 $N(w,z,\phi)$ & iff & the clause $w$ contains the negative literal $\lnot z$ in the CNF encoding\\
&& of the propositional formula $\phi$\\
 $\bar{E}(z)$ & iff & the propositional variable $z$ has been assigned the value 1, i.e. \bf{true}\\
\end{tabular}

The corresponding propositional formula $\phi$ is satisfiable iff the 
following second order existential formula is true:
\begin{equation}
\label{SATSOE}
\exists \bar{E} \forall w \exists z (P(w,z,\phi) \land \bar{E}(z)) \lor (N(w,z,\phi) \land \lnot \bar{E}(z))
\end{equation}

\section{Satisfiability of bounded fragments of propositional logic}

Let us study the satisfiability of propositional formulas in CNF, with the encoding introduced
in the previous section. Let us denote an effective binary representation
of such an encoding of a propositional formula $\phi$ by $\theta(\phi)$.
We shall denote the satisfiability problem of propositional logic
by SAT. With the binary encoding at our disposal,
we can extend the domain of \eqref{SATSOE} to the set of all finite binary strings $\{y\}$ by
\begin{equation}
\label{SAT}
\begin{split}
\tilde{\Psi}_{\mathrm{SAT}} \leftrightarrow \\
\exists y \exists \phi \exists \bar{E} \forall w \exists z (y=\theta(\phi) \land \\
(P(w,z,\phi) \land \bar{E}(z)) \lor (N(w,z,\phi) \land \lnot \bar{E}(z))
\end{split}
\end{equation}

As an example of such an efficient binary encoding scheme $\theta(\cdot)$, let us consider the following
scheme.
\begin{enumerate}
\item The binary representation of a natural number indicating the index of a propositional variable is
encoded in every fourth bit of our input string $y$ only, starting from the fourth bit.
\item The preceding and intervening three-bit sequences are used as codes on interpreting the last bit, according to the following table
\end{enumerate}
\vspace{0.5cm}
\begin{tabular}{|l|l|}
\hline
Code & Meaning \\
\hline
000 & Next bit continues the binary representation of the index of the current \\
& propositional variable\\
001 & Next bit begins the binary representation of the index of a new propositional\\
&  variable in the current clause \\
010 & Next bit represents a propositional constant: 0 for {\tt false}, 1 for {\tt true} \\
011 & Next bit is the first bit in the binary representation of the index of the first variable\\
&  in a new clause\\
100 & Next bit is the first bit in the binary representation of the first variable in the first\\
&  clause in the list $N$ \\
101 & End of input string \\
\hline
\end{tabular}
\vspace{0.5cm}

Let us parameterize propositional formulas $\phi$ by the length $M$ of their encoding
in this effective binary encoding scheme $\theta(\cdot)$. Let us denote the corresponding fragment
of well-formed propositional formulas of encoding length at most $M$ by $\mathrm{L}_M$ and the corresponding
satisfiability problem by $\mathrm{SAT}_M$. The domain of definition of $\mathrm{SAT}_M$ is likewise extended to the set of all finite binary strings $\{ y \}$ of length at most $M$ by the formula $\tilde{\Psi}_{\mathrm{SAT}_M}$, with $|y|$
denoting the length of the binary string $y$.
\begin{equation}
\label{SAT_M}
\begin{split}
\tilde{\Psi}_{\mathrm{SAT}_M} \leftrightarrow \\
\exists y \exists \phi \exists \bar{E} \forall w \exists z (|y| \leq M \land y=\theta(\phi) \land \\ (P(w,z,\phi) \land \bar{E}(z)) \lor (N(w,z,\phi) \land \lnot \bar{E}(z))
\end{split}
\end{equation}
The condition $|y| \leq M$ is equivalent to the requirement $y \in 2^M$.
We can restrict the domain of the first order variables in \eqref{SAT_M} to be of size $2^{2^{2M}}$.

We have now two formulas in existential second order logic that can be used to capture SAT.
A crucial difference between  \eqref{SAT} and \eqref{SAT_M} is the fact that, unlike
$\tilde{\Psi}_{\mathrm{SAT}}$, $\tilde{\Psi}_{\mathrm{SAT}_M}$ can be recast as a first order formula for any fixed $M$. The
existentially quantified relation $\bar{E}(z)$ in this case ranges over a finite set of size at most $2^{M}$ only.
Hence the $SO\exists$ quantification $\exists \bar{E}$ 
can be replaced by a $FO$ quantification $\exists e$, after choosing a separate variable 
$e$ with a domain that comprises a finite set of truth assignments. 

We replace the unary relation $\bar{E}(z)$ with a binary relation $E(e,z)$ but retain the ternary relations $P(w,z,\phi)$  and $N(w,z,\phi)$ that have the following intended interpretations

\begin{tabular}{lll}
 $E(e,z)$ & iff the propositional variable $z$ has the value 'true' in the\\
& truth assignment $e$ \\
 $P(w,z,\phi)$ &iff the clause $w$ with a positive literal $z$ belongs to the set of positive\\
 &clause-variable pairs of the propositional formula $\phi$\\
 $N(w,z,\phi)$ &iff the clause $w$ with a negative literal $\lnot z$ belongs to the set of negative\\
 & clause-variable pairs of the propositional formula $\phi$\\
\end{tabular}

If we set a fixed finite bound on the size of our models, as in
\eqref{SAT_M}, the set of models of such a restriction of \eqref{SAT} becomes finite and the
corresponding theory primitive recursive by exhaustive search:  
\begin{equation}
\label{SATFO_M}
\begin{split}
\Psi_{\mathrm{SAT}_M} \leftrightarrow \\
\exists \phi \exists y \exists e \forall w \exists z (|y| \leq M \land y=\theta(\phi) \land \\
(P(w,z,\phi) \land E(e,z)) \lor (N(w,z,\phi) \land \lnot E(e,z)))
\end{split}
\end{equation}
We can now associate the $FO$-theory defined by $\Psi_{\mathrm{SAT}_M}$ with the bounded satisfiability problem $\mathrm{SAT}_M$, for each $M > 0$. 
The variable $e$ ranges over assignments of truth values to at most $2^M$ propositional
variables. The size of the domains of $z$, $w$ and $y$ is at most $2^M$ and
the size of the domain of $\phi$ at  most $2^{2^{2M}}$ and that of $e$ at most $2^{2^M}$. 

We shall use \textit{italics} in $L_M$ and $SAT_M$ to indicate the
first order language and theory of this finite fragment of the propositional satisfiability problem, respectively. 
We take $SAT_M$ to be closed under implication but not necessarily complete. We also allow $SAT_M$ to use finitely many first order variables up to some limit that can grow without bound with $M$. The union of all $SAT_M$'s is therefore not a first order theory, but all individual $SAT_M$'s are.

\section{Encoding deterministic Turing machines in first order logic}
Modifying the notation introduced by B\"{o}rger in \cite{Borger} to conform to a first-order, rather than
propositional, definition of an arbitrary Turing machine $T$, we first
define a program formula 
that any  $\mathrm{SAT}$ solving Turing machine $T$ is required to satisfy. We shall also define corresponding
input and accepting halting state formulas $y=\theta(\phi)$ and
$\omega$, respectively, and amend the program formula so that it corresponds to
a deterministic Turing machine, when appropriate.

The input formula 
\begin{equation}
\label{input}
y=\theta(\phi)
\end{equation} 
states that at the beginning
of the computation, the first positions to the right from the starting position of the input tape of the
Turing machine $T$ contain the binary encoding of the propositional formula
$\phi$ in the CNF encoding introduced in the Introduction. The input tape is read-only.
The input formula $y=\theta(\phi)$ implies that our Turing machine checks the syntax
of its input. Therefore it must also read all of its input.

The program formula  defines  time-limited
computations of an arbitrary Turing machine $T$, with a bound $b$ on the number of time steps taken. It has the form
\begin{equation}
\label{program}
 \pi_{T}(t,t',u,u',v,v',\tau,\bar{t},\bar{u},\bar{v},o_{\bar{t}},o_{\bar{u}},o_{\bar{v}})
\end{equation}
with variables to be described below. We shall explicitly assume
the time parameter in our formalization. Bound variables $t, t'$ stand for time steps, $u, u'$ for tape cells on the work tape 
and $v, v'$ for tape cells on the input tape, $\tau$ denotes the current time step,
 $\bar{t}$ the $b+1$-tuple of time steps $(t_0, \ldots , t_b)$,  $\bar{u}$ the $2b+1$-tuple of
working tape cells $(u_{-b}, \ldots , u_b)$, and  $\bar{v}$ the $2b+1$-tuple of input tape cells $(v_{-b}, \ldots , v_b)$.
The variables $o_{\bar{t}},o_{\bar{u}}$ and $o_{\bar{v}}$ denote permutations of the elements
that have been ordered in the vectors $\bar{t},\bar{u}$ and $\bar{v}$, respectively.

For arbitrarily large finite models, linear order is not $FO$-definable. 
However, a total order in a finite set of bounded size can be defined in $FO$ with a 
relation $\bar{O}(u,u')$ that satisfies the definition
\begin{equation}
\label{Order}
\forall u,u' (\lnot(\bar{O}(u,u') \land \bar{O}(u',u)) \land (\bar{O}(u,u') \lor \bar{O}(u',u) \lor u=u'))
\end{equation} 
To express the existence of such an order in a bounded set, we employ an analogous devise
to the one employed in \eqref{SATFO_M}. Noting that the set of permutations of a bounded set is itself a bounded set,
let us define a relation $Ord^{b+1}(o,u_0,\ldots,u_b)$ on a set of size $b+1$ by the sentence
\begin{equation}
\label{ord}
Ord^{b+1}(o,u_0,\ldots,u_b) \leftrightarrow \exists o \exists u_0, \ldots, u_b (O(o,u_0,u_1) \land \ldots \land 
O(o,u_{b-1},u_b)
\end{equation}
with altogether $\frac{b (b+1)}{2}$ appearances of a relation term of the form $O(o,u,u')$. 
Here the ternary relation $O(o,u,u')$ has the intended meaning that the corresponding
binary relation $\bar{O}(u,u')$ satisfies the definition \eqref{Order} for the permutation
$o$ of the $b+1$ elements in the valuation of the relation $Ord^{b+1}(o,u_0,\ldots,u_b)$. 
Hence $o$ ranges over a set of size $(b+1)!$. Let us further introduce the notation
$-b \leq u, u' \leq b$, by which we mean that
\begin{equation}
\begin{split}
-b \leq u, u' \leq b \leftrightarrow \\
(u=u_{-b} \lor \ldots \lor u=u_b) \land \\
(u'=u_{-b} \lor \ldots \lor u'=u_b) 
\end{split}
\end{equation}
and $0 \leq t, t' \leq b$, by which we mean that
\begin{equation}
\begin{split}
0 \leq t, t' \leq b \leftrightarrow \\
(t=t_{0} \lor \ldots \lor t=t_b) \land \\
(t'=t_{0} \lor \ldots \lor t'=t_b) 
\end{split}
\end{equation}

Let us define the successor relation $S(u,u',\bar{u},o_{\bar{u}})$ on the permutation $o_{\bar{u}}$ of any set of 
size $2b+1$ by
\begin{equation}
\label{successor}
\begin{split}
S(u,u',\bar{u},o_{\bar{u}}) \leftrightarrow \\
\forall (-b \leq v \leq b) (O(o_{\bar{u}},u,u') \land O(o_{\bar{u}},u,v)
\rightarrow O(o_{\bar{u}},u',v) \lor (v=u'))
\end{split}
\end{equation} 

Let us denote the set of permutations of the bounded set of time steps $\{t\}$ by $\{o_{\bar{t}}\}$, the
set of permutations of the set $\{u\}$ of working tape cells by $\{o_{\bar{u}}\}$ and the set of permutations of the 
set $\{v\}$ of input tape cells by $\{o_{\bar{v}}\}$. The sizes of the corresponding sets of permutations are 
$(b+1)!$ for the set $\{o_{\bar{t}}\}$  and $(2b+1)!$ for the sets $\{o_{\bar{u}}\}$ and $\{o_{\bar{v}}\}$, 
respectively, because input and work tapes can be traversed in either direction.

A bounded and finite set can always be totally ordered and a unique successor exists for all of its elements but the last one.
But we have not specified any ordering beforehand, since its existence is explicitly established in
the defining formula \eqref{ord}. The results below therefore apply to any bounded and finite set, and not just to
ordered structures. 

Returning to the program formula \eqref{program}, for each 
$t,t',u,u',v,v',\tau,\bar{t},\bar{u},\bar{v},o_{\bar{t}},o_{\bar{u}},o_{\bar{v}}$, the formula 
$\pi_{T} (t,t',u,u',v,v',\tau,\bar{t},\bar{u},\bar{v},o_{\bar{t}},o_{\bar{u}},o_{\bar{v}})$ 
is a first order formula over the relational vocabulary
$A(t,u)$, $B_l (t,u)$, $B_{in} (t_0,v)$, $Z_q (t)$, $A_{in} (t,v)$, $S(t,t',\bar{t},o_{\bar{t}})$, 
$S(u,u',\bar{u},o_{\bar{u}})$ and $S(v,v',\bar{v},o_{\bar{v}})$. 
Parameter $l$ stands for a letter in the alphabet and $q$ for a
$T$-state. We shall take the following intended interpretations for
the relations in our vocabulary:

\begin{tabular}{lll}
 $A(t,u)$ & iff & working cell on work tape at time $t$ is $u$ \\
 $A_{in}(t,v)$ & iff & reading cell on input tape at time $t$ is $v$ \\
 $B_l (t,u)$  & iff & letter $a_l$ is in cell $u$ of the work tape at time $t \geq 0$ \\
 $B_{in} (t_0,v)$  & iff & bit 1 is in cell $v$ of the input tape at time $t_0$, and thereafter \\
 $Z_q(t)$ & iff & $q$ is the state of $T$ at time $t$, with $Z_{\omega}(t)$ indicating that $T$ has halted\\
 $S(t,t',\bar{t},o_{\bar{t}})$  & iff & $t'$ is the direct successor instant of $t$ in the tuple $\bar{t}$,\\
 && totally ordered in the permutation $o_{\bar{t}}$ \\
 $S(u,u',\bar{u},o_{\bar{u}})$  & iff & $u'$ is the direct successor cell of $u$ on the work tape in the tuple $\bar{u}$,\\
 && totally ordered in the permutation $o_{\bar{u}}$ \\
 $S(v,v',\bar{v},o_{\bar{v}})$  & iff & $v'$ is the direct successor cell of $v$ on the input tape in the tuple $\bar{v}$,\\
 && totally ordered in the permutation $o_{\bar{v}}$ \\
\end{tabular}

A \textit{correct Turing machine configuration} is a description of
an instantaneous state of a Turing machine $T$ at any given time step $\tau$. 

For any $T$-configuration, let 
$\pi_{T}(t,t',u,u',v,v',\tau,\bar{t},\bar{u},\bar{v},o_{\bar{t}},o_{\bar{u}},o_{\bar{v}})$ at time $\tau$
feature the conjunction of the above literals describing that
configuration. The domain of this program formula is the Cartesian product set
\begin{equation}
N = \{t\}^3 \times \{u\}^2 \times \{v\}^2 \times \{ \bar{t}\} \times \{ \bar{u}\}\times \{ \bar{v}\} 
\times \{o_{\bar{t}}\} \times \{o_{\bar{u}}\} \times \{o_{\bar{v}}\}
\end{equation} 
At time $\tau = t_b$, the size $|N|$ of this domain is bounded by $(b+1)^3 \cdot (2 b + 1)^4 \cdot
(b+1)^{b+1} \cdot (2b+1)^{2b+1} \cdot (b+1)! \cdot ((2b+1)!)^2$.

The program formula $\pi_{T}(t,t',u,u',v,v',\tau,\bar{t},\bar{u},\bar{v},o_{\bar{t}},o_{\bar{u}},o_{\bar{v}})$ 
defines the logical rules that correspond to the program steps in the program of $T$,
so that its models simulate $T$-computations in the
sense of the following \textit{Simulation Lemma}, modified from \cite{Borger}, p.
480:
\begin{lemma}
\label{sim}
Let $(\mathring{t},\mathring{t'},\mathring{u},\mathring{u'},\mathring{v},\mathring{v'},\mathring{\tau},
\mathring{\bar{t}},\mathring{\bar{u}},\mathring{\bar{v}},
\mathring{o}_{\bar{t}},\mathring{o}_{\bar{u}},\mathring{o}_{\bar{v}})$ 
be a valuation, i.e. an assignment of an element of $N$ to the tuple 
$(t,t',u,u',v,v',\tau,\bar{t},\bar{u},\bar{v},o_{\bar{t}},o_{\bar{u}},o_{\bar{v}})$. 

Models of $\pi_{T}(t,t',u,u',v,v',\tau,\bar{t},\bar{u},\bar{v},o_{\bar{t}},o_{\bar{u}},o_{\bar{v}})$ 
for any fixed time $\tau$ are sets of such valuations for which the program formula is valid. 
For arbitrary $T$-configurations at time $\tau=t_0$, with the corresponding program formula
$\pi_{T}(t,t',u,u',v,v',t_0,\bar{t},\bar{u},\bar{v},o_{\bar{t}},o_{\bar{u}},o_{\bar{v}})$ and for an arbitrary future time $t_k$, where $t_0 < t_k \leq t_b$, if 
$(\mathring{t},\mathring{t'},\mathring{u},\mathring{u'},\mathring{v},\mathring{v'},\mathring{\tau},\mathring{\bar{t}},
\mathring{\bar{u}},\mathring{\bar{v}},
\mathring{o}_{\bar{t}},\mathring{o}_{\bar{u}},\mathring{o}_{\bar{v}})$ satisfies 
$\pi_{T}(t,t',u,u',v,v',t_0,\bar{t},\bar{u},\bar{v},o_{\bar{t}},o_{\bar{u}},o_{\bar{v}})$ at time $t_0$, 
then for at least one $T$-configuration 
$\pi_{T}(t,t',u,u',v,v',t_k,\bar{t},\bar{u},\bar{v},o_{\bar{t}},o_{\bar{u}},o_{\bar{v}})$ at time $t_k$
which satisfies 
\begin{equation}
\pi_{T}(t,t',u,u',v,v',t_0,\bar{t},\bar{u},\bar{v},o_{\bar{t}},o_{\bar{u}},o_{\bar{v}}) \vdash ^k_T 
\pi_{T}(t,t',u,u',v,v',t_k,\bar{t},\bar{u},\bar{v},o_{\bar{t}},o_{\bar{u}},o_{\bar{v}})
\end{equation}
 $(\mathring{t},\mathring{t'},\mathring{u},\mathring{u'},\mathring{v},\mathring{v'},\mathring{\tau},
\mathring{\bar{t}},\mathring{\bar{u}},\mathring{\bar{v}},
\mathring{o}_{\bar{t}},\mathring{o}_{\bar{u}},\mathring{o}_{\bar{v}})$ 
also satisfies $\pi_{T}(t,t',u,u',v,v',t_k,\bar{t},\bar{u},\bar{v},o_{\bar{t}},o_{\bar{u}},o_{\bar{v}})$ at time $t_k$.
\end{lemma}

The notation above,  
\begin{equation}
\pi_{T}(t,t',u,u',v,v',t_0,\bar{t},\bar{u},\bar{v},o_{\bar{t}},o_{\bar{u}},o_{\bar{v}}) \vdash ^k_T 
\pi_{T}(t,t',u,u',v,v',t_k,\bar{t},\bar{u},\bar{v},o_{\bar{t}},o_{\bar{u}},o_{\bar{v}})
\end{equation}
means that the Turing
machine configuration $\pi_{T}(t,t',u,u',v,v',t_k,\bar{t},\bar{u},\bar{v},o_{\bar{t}},o_{\bar{u}},o_{\bar{v}})$ at time $t_k$  is a successor configuration to 
$\pi_{T}(t,t',u,u',v,v',t_0,\bar{t},\bar{u},\bar{v},o_{\bar{t}},o_{\bar{u}},o_{\bar{v}})$ at time $t_0$,
and follows from it through the execution of $k$ steps of the program
of the Turing machine $T$. This result is achieved by representing
every state transition of $T$ as a step of logical inference.

As an example, for a program step that involves a state transition
from state $q$ to state $q'$ if there is the letter $a_l$ on the
work tape at the cell where the read/write head resides, writing the
letter $a_{l'}$ on that tape cell, and backing the read/write head
one step back to the left, the corresponding part of $T$'s
program formula $\pi_{T}(t,t',u,u',v,v',\tau,\bar{t},\bar{u},\bar{v},o_{\bar{t}},o_{\bar{u}},o_{\bar{v}})$ reads:
\begin{equation}
\label{Turing_step}
\begin{split}
(Z_q(t) \land A(t,u) \land B_l (t,u) \land S(t,t',\bar{t},o_{\bar{t}}) \land S(u',u,\bar{u},o_{\bar{u}})) \rightarrow \\
(Z_{q'}(t') \land A(t',u') \land B_{l'} (t',u))
\end{split}
\end{equation}
and analogously for all instructions in the program of the Turing machine.
As part of the program formula $\pi_{T}(t,t',u,u',v,v',t_0,\bar{t},\bar{u},\bar{v},o_{\bar{t}},o_{\bar{u}},o_{\bar{v}})$, 
the initial configuration
of the Turing machine $T$ at time $t_0$ is also defined. This means the definition that the work
tape is empty, the Turing machine is in its initial state and that both read/write heads are at
positions $u_0$ and $v_0$, respectively. However, the program formula does $not$ define
the contents of the input tape. Instead, the free variable $y$ of the input formula \eqref{input} that ranges over values
$\{0, \ldots ,2^{b}\}$, as expressed by the binary input string on the input tape, defines the input to $T$.

The absolute values of the tape variables $u$ and $v$ are always
bounded by $t$, as the read/write head can only move at most one
cell to the left or to the right in a single time step. The
computations of the Turing machine $T$ up to any time bound $b$ are therefore defined by the
\textit{bounded computation formula}
\begin{equation}
\begin{split}
\label{Turing}
\Pi^b_{T} \leftrightarrow Ord^{b+1}(o_{\bar{t}},\bar{t}) \land Ord^{2b+1}(o_{\bar{u}},\bar{u}) \land Ord^{2b+1}(o_{\bar{v}},\bar{v}) \land  \exists y \land\\
\forall (0 \leq t,t',\tau \leq b) \forall (-b \leq u,u' \leq b) \forall (-b \leq v,v' \leq b) \\
\pi_{T}(t,t',u,u',v,v',\tau,\bar{t},\bar{u},\bar{v},o_{\bar{t}},o_{\bar{u}},o_{\bar{v}})
\end{split}
\end{equation}
The condition $\exists y$ is an expression of the requirement that the input tape specified by
a permutation of the input cells defined by $\bar{v}$ contains some binary string.

Because of the time bound $b$, the formula \eqref{Turing} is a first order formula. 
In this notation, the halting state formula $\omega$ can be expressed as
\begin{equation}
\omega \leftrightarrow \bigvee_{i=0}^{b} Z_{\omega} (t_i)
\end{equation}
up to any upper time limit $b$, or equivalently as
\begin{equation}
\omega \leftrightarrow Z_{\omega} (t_b)
\end{equation}
because any Turing machine will stay in the halting state indefinitely, once it has been entered, and this
condition is also stated in the program formula of $T$.

Let us abbreviate the tuple $(t,t',u,u',v,v',\tau)$ by $\bar{x}$, so that we
implicitly bundle together our tuple of seven 'scalar'  bound
variables in \eqref{Turing}, the absolute values of which are all bounded by time $t$. By the notation $\bar{x}\leq b$ we shall mean
that
\begin{equation}
\begin{split}
\bar{x}\leq b \leftrightarrow \\
(t=t_0 \lor \ldots \lor t=t_b) \land \\
(t'=t_0 \lor \ldots \lor t'=t_b) \land \\
(u=u_{-b} \lor \ldots \lor u=u_b) \land \\
(u'=u_{-b} \lor \ldots \lor u'=u_b) \land \\
(v=v_{-b} \lor \ldots \lor v=v_b) \land \\
(v'=v_{-b} \lor \ldots \lor v'=v_b) \land \\
(\tau=t_0 \lor \ldots \lor \tau=t_b)
\end{split}
\end{equation}

The bounded computation formula \eqref{Turing} that defines $\Pi^b_{T}$
can be augmented to always simulate a deterministic Turing machine
by adding a uniqueness formula. The uniqueness formula states that
the validity of any one of the formulas of type  \eqref{Turing_step}
that states the transition from state $Z_i (t)$ to state $Z_{i'}
(t')$ between successive time steps $t$ and $t'$ implies the
negation of every other similar state transition formula from $Z_i
(t)$ to $Z_j (t')$, when $j \neq i'$, and
conjuncting this formula for every pair of states $(i,j)$. Similar formulas
for uniqueness of read/write head movements to the left or to
the right, and for imposing uniqueness of letters printed at the working cell of the work tape
need to be added, too.

There are at most a constant number of such uniqueness formulas per
time step, since there are at most a constant number of
different states and letters in the alphabet of the Turing machine
$T$. Let us denote the number of states of $T$ by $|Q|$. Analogous uniqueness formulas are present in 
\eqref{program} already in the non-deterministic case
for all tape cells, up to a number reachable in $b$ time steps,
requiring the uniqueness of the presence of any letter in any one
tape cell at each time step; the uniqueness of the working or reading cell at each time step
and the uniqueness of the successor relation between time steps
and tape cells, as indicated in\eqref{ord} and \eqref{successor}. The number of uniqueness formulas of successor relations
grows quadratically with the time bound $b$ but the rest linearly, since the 
number of letters and machine states are fixed.

Let us now modify the bounded computation formula \eqref{Turing} so that it will
define a bounded set of computations by $T$ that includes all accepting computations on input
strings at most $M$ bits long as a subset. We shall denote by $b(M)$ the lowest bound on the maximum length
of computations needed to accept all the satisfiable formulas in the corresponding finite fragment of $\mathrm{L}$,
i.e. $\mathrm{L}_M$. This is the requirement of worst case complexity that is implicit in decision problems. The bounded computation formula that defines this bounded set 
of computations of our Turing machine is denoted by $\Pi^{b(M)}_T$, and defined as
\begin{equation}
\label{PiM}
\begin{split}
\Pi^{b(M)}_T \leftrightarrow Ord^{b(M)+1}(o_{\bar{t}},\bar{t}) \land Ord^{2b(M)+1}(o_{\bar{u}},\bar{u}) \land Ord^{2b(M)+1}(o_{\bar{v}},\bar{v}) \land \exists y \land\\
\forall (\bar{x} \leq b(M)) \pi_{T}(t,t',u,u',v,v',\tau,\bar{t},\bar{u},\bar{v},o_{\bar{t}},o_{\bar{u}},o_{\bar{v}})
\end{split}
\end{equation}

From a semantic perspective, parameterizing the bound on the set of computations
by input length instead of the number of time steps is not an essential restriction, as is stated in the following theorem.
\begin{theorem}
\label{timebound}
\begin{equation}
\exists(b>0)\Pi^b_{T} \leftrightarrow \exists (M>0) \Pi^{b(M)}_T
\end{equation}
\end{theorem}
\begin{proof} The reverse implication holds, because there is a finite number of acceptable inputs
of length at most $M$ bits and we can choose a $b$ that is sufficient for accepting them all.
For the direct implication, let us take the longest input string $y_{max}$ that our Turing machine accepts
by time $t_b$. This string always exists and is of length at most $b$, 
since we assume that our Turing machine always reads its input.
It then follows that $T$ accepts all those strings within bounded time.
Let us choose $M=|y_{max}|$. Since $\mathrm{L}_{|y_{max}|}$ is a finite fragment of
propositional logic, there exists a time bound $b(|y_{max}|)$ by which all satisfiable formulas in this fragment
have been accepted.  Hence the membership of any input string $y$ in $SAT$ is decided by $\Pi^{b(|y_{max}|)}_T$.
\end{proof}

Theorem \ref{timebound} states that a bound on execution time always implies a bound
on input length, and vice versa, for a $\mathrm{SAT}$ solving Turing machine that is required to read all its input.
The domain of the revised bounded computation formula \eqref{PiM} is included in the domain $N$
of the time-bounded program formula \eqref{program} when $b(M) \leq b$.

\section{Bounded Turing structures}

Let us now turn to models of the revised bounded computation formula \eqref{PiM}, i.e. to
$\Pi^ {b(M)}_T$-structures. These are directed acyclic graphs (DAGs) of computations
that satisfy the program formula \eqref{program} as they proceed from time step $t_0$ to
time step $t_{b(M)}$. $\Pi^ {b(M)}_T$-structures are not trees, because our Turing machine
may revert to the same state from different preceding computations. But since the set of time steps
from $t_0$ to $t_{b(M)}$ is totally ordered, $\Pi^ {b(M)}_T$-structures cannot loop back
in time, and they therefore do not contain cycles.

It is important to note that $\Pi^{b(M)}_T$ defines \textit{all} possible
Turing computations, halting or non-halting, of a given Turing
machine $T$ on \textit{all} binary input strings $y_0$
of length at most $M$. In addition, it may possibly define some computations on longer
input strings and also the initial segments of all other computations. 
The formula $\Pi^{b(M)}_T \land \omega$, on the other hand, 
defines all accepting computations that halt by time step $t_{b(M)}$, whether the input string is of length $M$
or longer - but in any case not longer than $b(M)$ bits. Only a
conjunction with an input formula \eqref{input} expressed in $\Pi^{b(M)}_T \land y = y_0$, i.e. specifying a particular
input string $y_0$, will yield a formula that defines just the
computations pertaining to a particular input. A further conjunction
with the halting formula $\omega$ will yield a formula that defines all halting
computations on a particular input: $\Pi^{b(M)}_T \land \omega \land
y=y_0$.

The Directed Acyclic Graph that corresponds to the bounded computation formula
 $\Pi^{b(M)}_T$ by itself serves as the "monster model" for all correct computations by $T$
up to time step $t_{b(M)}$. Conjunction of  $\Pi^{b(M)}_T$ with $\omega$ or with
various subsets of input strings $\bigvee_{i \in I} (y = y_i)$, where $I \subseteq \{0, \ldots ,2^M\}$
define sub-DAGs that are embedded in the monster-DAG defined by  $\Pi^{b(M)}_T$.
Because there is a bound $b(M)$ on the length of any path in the monster-DAG,
all paths in it terminate in leaf nodes.

For both deterministic and non-deterministic Turing machines,
the semantics of the associated first order bounded computation formula $\Pi^{(b(M)}_T \land
y=y_0$ are fully determined as soon as the contents of the
input tape have been determined by the input formula $y=y_0$.
For a non-deterministic Turing machine, even a fixed input string can correspond
to a "proper" DAG with multiple branches. For a deterministic Turing machine,
the sub-DAG that corresponds to a single input string is always a chain.
If all the branches emanating from a node terminate in the
accepting halting state $\omega$ by time step $t_{b(M)}$, let us call such a
sub-DAG a \textit{halting sub-DAG}.

\section{Atoms}

From now on, we shall assume that our Turing machine $T$ is
deterministic. Let us define an $Atom$ as a time dependent program
state $Z_q(t)$ that satisfies the following first order formula for 
binary string variables $y'$, $y''$ and $y'''$ representing the content of the input tape:
\begin{equation}
\label{atom}
\begin{split}
Atom^M(t,q) \leftrightarrow \exists y' \exists y'' \exists y''' \exists t' \exists u \\
(t_0 \leq t, t' \leq t_{b(M)} \land u_{-b(M)} \leq u \leq u_{b(M)} \land 1 \leq q \leq |Q| \land \\
((\Pi_T^{b(M)} \land (Z_q(t)) \rightarrow Z_{\omega}(t_{b(M)})) \land \\
(\Pi_T^{b(M)} \land y=y' \land |y'| \leq M \land Z_q(t)) \land \\
(\Pi_T^{b(M)} \land y=y'' \land |y''| \leq M \land \\ 
S(t',t,\bar{t},o_{\bar{t}}) \land \bigvee_{q'} \bigvee_l ((Z_{q'}(t') \land A(t',u) \land B_l(t',u)
\land y=y'') \land \\
((\Pi_T^{b(M)} \land Z_{q'}(t') \land A(t',u) \land B_l(t',u)
\land y=y'') \rightarrow Z_q(t)) \land \\
(\Pi_T^{b(M)} \land y=y''' \land |y'''| \leq M \land Z_{q'}(t') \land \\
(\Pi_T^{b(M)} \land y=y''' \rightarrow \neg Z_{\omega}(t_{b(M)})))))))
\end{split}
\end{equation}
$Atom$ is shorthand for \textit{Accepting Transition Of Machine-state}. 
Intuitively, the formula \eqref{atom} states that, within $\mathrm{L}_M$ and with any SAT
solving deterministic Turing machine $T$ with the bounded computation formula
$\Pi_{T}^{b(M)}$, 
\begin{itemize}
\item All computations on any input that assume state $q$ at
time $t$ end up in the accepting halting state by time step
$t_{b(M)}$; 
\item That there is at least one input $y'$ of length at most $M$ whose computation
attains state $q$ at time $t$; 
\item That for at least one state $q'$ during the previous time step
$t'$, from which a transition to state $q$ at time step $t$ is
carried out on some input $y''$ of length at most $M$ by the Turing machine $T$, there is at least
one other input $y'''$ of length at most $M$ that also assumes state $q'$ at time $t'$, but 
that leads to the rejection of that input, since by time step $t_{b(M)}$
all acceptable inputs of length $M$ or less will have been accepted. The
ones that have not halted by time step $t_{b(M)}$, despite their input being at most $M$ bits long,
will therefore never halt.
\end{itemize}

The property described above means that by time step $t_{b(M)}$, our Turing machine $T$ has
\textit{decided} $\mathrm{SAT}_M$.
In terms of  $\Pi^{b(M)}_T$-structures,
$Atoms$ correspond to root nodes of halting sub-DAGs in the monster model
defined by  $\Pi^{b(M)}_T$. There may be more than one $Atom$ in a halting sub-DAG. 
Any eventual state transition to $Z_{\omega}(t)$ at time $t$
comprises an $Atom$ by itself, except when it is encountered only on paths within a halting sub-DAG with an 
$Atom$ already preceding it on every path that leads to it.

One is tempted to call the set of $Atom$s the "event horizon" of our DTM, because it represents the set of
pairs of a state and a time step in the "spacetime" $\{q\} \times \{t\}$ of our DTM, beyond which all computation paths will end up in the "singularity" of the accepting halting state $\omega$, from which there is no return.
All paths to the halting state pass through some $Atom$, for any computation by $T$.
Let us call the first $Atom$ on the path of the computation by $T$ on the formula $\phi$ that belongs to $SAT_M$
the \textit{deciding Atom} of $\phi$. 
The set of $Atoms$ depends
on the bound on input length $M$, and the relation symbol $Atom^M(t,q)$ is
therefore equipped with the parameter $M$.

Since computations by the Turing machine $T$ are uniform, $\Pi_T^{b(M')}$-structures
for $M' > M$ will grow all the branches of their DAGs from the leaf nodes of
the current $\Pi_T^{b(M)}$-structure, our "monster model". Hence all halting sub-DAGs
in it will remain halting sub-DAGs in subsequent $\Pi_T^{b(M')}$-structures for all $M' > M$,
as all their leaf nodes are already in the halting state $\omega$ at time step $t_{b(M)}$,
and can never leave that state in any subsequent computations. Consequently, all $M$-$Atoms$
will stay on as $M'$-$Atoms$ for all $M' > M$. The set of $Atoms$ is therefore a monotonously increasing
set as a function of $M$. 

For any fixed finite bound on input length $M$, we can now establish a set of equivalences between
the following first order sentences
\begin{theorem}
\label{equiv}
\begin{equation}
\label{Cook_ext}
\begin{split}
\Psi_{\mathrm{SAT}_M} \leftrightarrow \exists \phi \exists y \exists e \forall w \exists z (|y| \leq M \land y=\theta(\phi) \land \\
(P(w,z,\phi) \land E(e,z)) \lor (N(w,z,\phi) \land \lnot E(e,z))) \leftrightarrow \\
\exists y \exists (t_0 \leq t \leq t_{b(M)}) \exists (1 \leq q \leq |Q|) (\Pi_T^{b(M)} \land |y| \leq M \land Z_q(t) \land Atom^{M}(t,q) )) \leftrightarrow \\
\exists y \Pi_T^{b(M)} \land |y| \leq M \land \omega 
\end{split}
\end{equation}
\end{theorem}
\begin{proof} The first formula defines $SAT_M$. By the construction of $\Pi_T^{b(M)}$
in section three, our Turing machine must halt at every input $y=\theta(\phi) \land \phi \in \mathrm{SAT}_M$
by time step $t_{b(M)}$, which is the statement of the third formula.
As to the second formula, the unique computation on the encoding of any satisfiable
propositional formula will pass through at least one $Atom$, and the first one of those
will be the one that decides that satisfiable formula. Therefore the time $t$ in $Atom^M(t,q)$
must satisfy $t \leq b(M)$, for any $\phi$ in $SAT_M$. On the other hand, every $Atom^M(t,q)$ will decide at least one propositional formula $\phi$ in $\mathrm{SAT}_M$ which confirms the first equivalence in \eqref{Cook_ext}.
The second equivalence follows from the definition of $Atom^M(t,q)$ as a node in a halting sub-DAG
that halts by time step $b(M)$, and the fact that we have restricted the last formula to inputs at most $M$ bits long.
\end{proof}

On the other hand, since any propositional formula has a finite encoding length, say $M'$, in our
effective encoding scheme, \eqref{Cook_ext} states that any accepting
computation of any $\mathrm{SAT}$ solving deterministic Turing machine 
implies that its input that encodes a propositional formula in $\mathrm{L}_{M'}$ 
belongs to $\mathrm{SAT}_{M'}$ for some $M'>0$. This implies
that the Turing machine halts by time step $t_{b(M')}$ on this input, which in turn implies
the passing of the corresponding computation through some $Atom^{M'}$
in $\Pi_T^{b(M')}$. When we allow $M$ to grow without limit, Theorem \eqref{Cook_ext} will 
therefore apply to every $\phi$ in SAT.

\section{Almost saturated models of $SAT_M$}

Let us now take a closer look at the models of the $FO$ theory $SAT_M$ that defines propositional 
satisfiability of CNF-formulas with an encoding at most $M$ bits long. The basics of model
theory used here can be captured from \cite{ChangKeisler} and \cite{Marker}. 
Starting from the defining formula \eqref{SATFO_M}
of $SAT_M$ on an individual input $y_0$, we see that each $SAT_M$ is a theory
that is defined by a disjunction of all those formulas \eqref{SATFO_M} on those inputs $y$
encoded by at most $M$ bits that correspond to satisfiable propositional formulas in CNF. 

When we take any complete extension of $SAT_M$, our first order language $L_M$ also contains formulas of the form shown below, with any number of variables $e_i$ up to some maximal finite number $\mathrm{maxsize}$ for each $M$, so that for each $1 \leq m \leq \mathrm{maxsize}$ we can define the truth of the sentence below. 
\begin{equation}
\label{saturation}
\begin{split}
\tilde{\eta}_m \leftrightarrow \exists e_1,\ldots, e_m  \exists \phi \exists y \bigwedge_{1 \leq i \leq m} \forall w \exists z (|y| \leq M \land y=\theta(\phi) \land \\
(P(w,z,\phi) \land E(e_i,z)) \lor (N(w,z,\phi) \land \lnot E(e_i,z)) \land \bigwedge_{1 \leq j < i \leq m} \lnot (e_j = e_i))
\end{split}
\end{equation}
From among the $\mathrm{maxsize}$ formulas of the type \eqref{saturation}, only the ones that have a propositional
formula $\phi$ in $\mathrm{SAT}_M$ with exactly $m$ satisfying truth assignments can be used to define
types in the $FO$ theory $SAT_M$. If this is the case, we call the propositional model size $m$ \textit{definable} in $SAT_M$. Formulas of type \eqref{saturation} grade propositional formulas $\phi$ featuring in $SAT_M$ into a decreasing chain by the number of satisfying truth assignments each of the corresponding $FO$ formulas of the form \eqref{saturation} admits, but they do not discern between individual propositional truth assignments. 

Let us now establish a lower bound on the size of any {\it almost saturated}
model of $SAT_M$, derived from the equivalences between $FO$ formulas in \eqref{Cook_ext}
for any fixed $M$. By an almost saturated model we mean a model that realizes an isolated type for
every different propositional model size definable in $\Psi_{\mathrm{SAT}_M}$.
\begin{theorem} 
\label{manymodels}
$SAT_M$ defines at least
$2^{\frac{1}{4k} M / \log M}$ different propositional model sizes for some constant $k>0$.
\end{theorem}
\begin{proof} Since $\mathrm{L}_M$ formulas
may contain different numbers of propositional variables, we need to define their propositional model
size after each formula has been complemented with dummy variables, up to the
maximum number of variables present in any $\mathrm{L}_M$-formula.

To get a lower bound on the number of propositional model sizes, let us look at increasing chains of sets of
propositional model sizes definable in $\mathrm{L}_M$. For that purpose,  let us take the following propositional formula
that defines an order relation $y > x$ between two binary strings 
$y = (y_{n-1}, y_{n-2}, \ldots , y_0)$ and $x = (x_{n-1}, x_{n-2}, \ldots , x_0)$:
\begin{equation}
\label{model_sizes}
\begin{split}
((y_{n-1}) \land (\lnot (x_{n-1}))) \lor & \\
(((y_{n-1}) \land (x_{n-1})) \lor ((\lnot (y_{n-1})) \land (\lnot
(x_{n-1}))))
\land  (((y_{n-2}) \land (\lnot (x_{n-2}))) \lor & \\
(((y_{n-2}) \land (x_{n-2})) \lor ((\lnot (y_{n-2})) \land (\lnot
(x_{n-2})))) \land  (((y_{n-3})
\land (\lnot (x_{n-3}))) \lor & \\
\cdots & \\
(((y_{1}) \land (x_{1})) \lor ((\lnot (y_{1})) \land (\lnot
(x_{1})))) \land ((y_{0}) \land (\lnot (x_{0}))) \ldots ) &
\end{split}
\end{equation}
Let us choose $n$ so that \eqref{model_sizes} is in $\mathrm{L}_M$. When the
binary string $y$ is fixed by a complete truth assignment to the
propositional variables ${y_{n-1}, y_{n-2}, \ldots , y_0}$, the
formula \eqref{model_sizes} defines all the truth assignments that
represent binary strings preceding $y$ in the numerical ordering
of binary numbers $x<y$. There are as many different such truth
assignments as is the cardinal denoted by the binary number $y$.

Since $y$ can take $2^n$ different values, the theory defined by the
formula \eqref{model_sizes} with all different binary strings for $y$,
interpreted as binary numbers, has
$2^n-1$ non-empty propositional models of different finite cardinality. By fixing
$y$ to all the different binary strings less than or equal to $2^n$ in turn,
while keeping the bits of $x$ as free propositional variables,
we get a family of $\mathrm{L}_M$ formulas that define $2^n-1$ different propositional model
sizes larger than zero. These can be picked up in any order to form an increasing
sequence of propositional model sizes of length $2^n-1$. 

It remains to compute the relation between $n$ and $M$. The formula
\eqref{model_sizes} has a length linear in $n$, since every free
propositional variable $x_i$ appears in it exactly three times and
there are nine Boolean operations and 14 pairs of parentheses per
variable. The encoding of the variables takes $\log n$ bits. If we
choose $M$ to be at least $\frac{1}{4} n \log n$, which is asymptotically
enough to cover the length of the encodings of the Boolean operations and the multiplicity
three of the variables in \eqref{model_sizes}, and a factor $k$ for encoding
\eqref{model_sizes} in CNF, we get the claim of the theorem.
\end{proof}

Let us now show that any model of a complete extension of $SAT_M$ must realize all propositional model sizes definable in \eqref{SATFO_M}. 
\begin{theorem}
\label{sizes}
Any model of a complete extension of $SAT_M$ must be almost saturated.
\end{theorem}
\begin{proof}
If one of the propositional model sizes were omitted by a model of a complete extension of \eqref{SATFO_M}, i.e. that there is no element in the model that separates between some pair of  definable propositional model sizes, then
for those definable propositional model sizes $m$ and $m+l$ and some $\phi \in SAT_M$ with $m$ satisfying truth assignments, the formula
\begin{equation}
\begin{split}
\tilde{\eta}_m \leftrightarrow \exists e_1,\ldots, e_m \exists \phi \exists y \bigwedge_{1 \leq i \leq m} \forall w \exists z 
(|y| \leq M \land y=\theta(\phi) \land \\
(P(w,z,\phi) \land E(e_i,z)) \lor (N(w,z,\phi) \land \lnot E(e_i,z)) \land \bigwedge_{1 \leq j < i \leq m} 
\lnot (e_j = e_i)) \leftrightarrow \\
\tilde{\eta}_{m+l} \leftrightarrow \exists e_1, \ldots, e_{m+l}  \exists \phi \exists y  \bigwedge_{1 \leq i \leq m+l} \forall w \exists z (|y| \leq M \land y=\theta(\phi) \land \\
(P(w,z,\phi) \land E(e_i,z)) \lor (N(w,z,\phi) \land \lnot E(e_i,z)) \land \bigwedge_{1 \leq j < i \leq m+l} \lnot (e_j = e_i))
\end{split}
\end{equation} 
would be true, even if $\phi$ does not have even $m+1$ satisfying truth assignments and this leads to a contradiction.
\end{proof}

A set of isolated types of a complete extension of $SAT_M$ can be defined by the formulas
\begin{equation}
\label{types}
\eta_m \leftrightarrow\tilde{\eta}_m \land \lnot \tilde{\eta}_{m+1} 
\end{equation}
for all model sizes $ m \in \{ 1, \ldots, \mathrm{maxsize} \} $ definable in $SAT_M$.
Let us now derive a lower bound on the size of any model of the $FO$ theory $SAT_M$ defined by \eqref{SATFO_M}.
\begin{theorem}
\label{modelsize}
Any model of \eqref{SATFO_M} has at least size $2^{\frac{1}{4k} M / \log M}$ for some constant $k>0$.
\end{theorem}
\begin{proof} 
A model of \eqref{SATFO_M} may omit an isolated type of its complete extension if this type is indiscernible in the set of
instances of bounded propositional satisfiability problem $\mathrm{SAT}_M$. We shall show that all definable
propositional model sizes are discernible in the set of instances of $\mathrm{SAT}_M$.

Let us recall that our encoding scheme does not increase encoding length when we replace all appearances of propositional variables $z$ in the encoding $\theta(\phi)$ of the propositional
formula $\phi$ by the constants 0 or 1, according to the satisfying truth assignment defined by $e$.
Let us call this property of an encoding scheme for propositional satisfiability the {\it valuation
property}. Let us denote the propositional formula that results from the replacement of all propositional literals appearing in $\phi$ by the corresponding constants 0 and 1, according to the truth assignment $e$, by $\phi(e)$.
Because of the valuation property, we have
\begin{equation}
\forall \phi \forall e (\phi \in \mathrm{L}_M \rightarrow \phi(e) \in \mathrm{L}_M)
\end{equation}
and all formulas $\phi(e)$ are therefore also decided by $T$ by time step $b(M)$ whenever $\phi \in \mathrm{L}_M$.
All formulas \eqref{saturation} are faithfully represented in $\mathrm{SAT}_M$ by the equivalence
\begin{equation}
\label{representation}
\begin{split}
\tilde{\eta}_m \leftrightarrow \exists e_1,\ldots, e_m  \exists \phi  \bigwedge_{1 \leq i \leq m}\exists y \forall w \exists z (|y| \leq M \land y=\theta(\phi(e_i)) \land \\
(P(w,z,\phi(e_i)) \land E(e_i,z)) \lor (N(w,z,\phi(e_i)) \land \lnot E(e_i,z)) \land \bigwedge_{1 \leq j < i \leq m} \lnot (e_j = e_i))
\end{split}
\end{equation}
Hence so are all formulas \eqref{types} and they define a set of distinct isolated types of $SAT_M$.

By Theorem \eqref{sizes} any model of a complete extension of \eqref{SATFO_M} must realize all definable propositional model sizes defined by \eqref{types}. By the formula \eqref{representation} all these types are discernible in the set of
instances of the bounded propositional satisfiability problem $\mathrm{SAT}_M$. Hence none of them can be
omitted by any model of \eqref{SATFO_M} either. 
We can have as many inequivalent formulas of the form \eqref{types} as there
are different definable propositional model sizes in $SAT_M$. These inequivalent formulas then each isolate a complete type of \eqref{SATFO_M}. By Theorem \eqref{manymodels}, there are at least $2^{\frac{1}{4k} M / \log M}$ different propositional model sizes definable in $SAT_M$ for some constant $k>0$.
\end{proof}

\section{A lower bound on the computational complexity of
propositional satisfiability}

By the set of equivalences in \eqref{Cook_ext}, we see that any model of the third
formula in \eqref{Cook_ext} must contain an isomorphic copy of a model of the second formula that must contain 
an isomorphic copy of a model of the first
formula which, by the requirement of worst case complexity, must be almost saturated,
i.e. contain an isolated type for each different propositional model size definable in $SAT_M$.

These model isomorphisms in \eqref{Cook_ext} are defined by the encoding $\theta(\phi)$ of each propositional formula $\phi$ and by the association of $\phi$ to its unique
deciding $Atom^M$ in the computational path that corresponds to $\theta(\phi)$, respectively. As we have restricted our set of equivalences in \eqref{Cook_ext} onto inputs $y$ at most $M$ bits long, we get the following theorem:
\begin{theorem}
\label{model_embeddings}
Any model of the formula $\exists y \Pi^{b(M)}_T \land |y| \leq M \land \omega$ that defines time-limited halting computations on all inputs at most $M$ bits long by time step $b(M)$ of a propositional satisfiability solving deterministic Turing machine $T$ must contain an isomorphic copy of a model of the formula $\Psi_{\mathrm{SAT}_M}$ that defines the first order theory $SAT_M$ that has at least as many elements as the number of different propositional model sizes definable in $SAT_M$, up to a a constant multiplier.  
\end{theorem}
\begin{proof}
By Theorem \eqref{equiv}, all models of any of the formulas in \eqref{Cook_ext} must be isomorphic. Since the
first formula defines $SAT_M$ and cannot have a model smaller than the number of different propositional model sizes
definable in $SAT_M$, this lower bound on model size carries over onto theories defined by the other two formulas as well. 
\end{proof}

From Corollary \eqref{modelsize} above we conclude that any model of any of the latter two formulas
in \eqref{Cook_ext} must contain an isomorphic copy of an almost saturated model of the first formula and therefore has a lower bound on its size  $2^{\frac{1}{4kd} M  / \log
M}$ as a function of $M$ for some constants $k, d > 0$, where $k$ accounts for the bound in Theorem \eqref{manymodels} and $d$ caters for our suboptimal
encoding of propositional formulas. This lower bound on
model size allows us to deduce the following theorem
on the deterministic time complexity of SAT.
\begin{theorem}
\label{SAT_complexity}
The deterministic time complexity of Propositional
Satisfiability $\mathrm{SAT}$ is not less than $C 2^{c M / \log M}$ for some
constants $C > 0$, $c > 0$ with respect to the length $M$ of the input
of a $\mathrm{SAT}$ solving deterministic Turing machine.
\end{theorem}
\begin{proof}
We pick the second formula in \eqref{Cook_ext}
\begin{equation}
\label{Cook_atom}
\begin{split}
\exists y \exists (t_0 \leq t \leq t_{b(M)}) \exists (1 \leq q \leq |Q|) (\Pi_T^{b(M)} \land |y| \leq M \land Z_q(t) \land Atom^M(t,q))
\end{split}
\end{equation}
The models of the theory defined by \eqref{Cook_atom} must have a model size at least 
$2^{\frac{1}{4kd} M / \log M}$. This model size is defined directly
by the number of its $Atom^M$'s. To satisfy the lower bound in Theorem \eqref{modelsize} it must have as many distinct $Atom^M$'s, all defined
by the subformula
$\exists (t_0 \leq t \leq t_{b(M)}) \exists (1 \leq q \leq |Q|) (Z_q(t) \land Atom^M(t,q))$
in \eqref{Cook_atom}. All terms of the form $Z_q(t) \land Atom^M(t,q)$ in
\eqref{Cook_atom} are positive. There cannot be more different
$Atom^M$'s than there are different propositional variables $Z_q(t)$. The program
of any Turing machine is, by uniformity, to be independent of the length
of its inputs. We therefore have at most a constant number $ |Q|$ of
different machine-states. This implies that in order for
\eqref{Cook_atom} to have a model size of at least $2^{\frac{1}{4kd} M / \log M}$,
as required by Corollary \ref{modelsize}, we need to use at least
$\frac{1}{|Q|} 2^{\frac{1}{4kd} M / \log M}$ time steps to supply all the $Atoms$ we need.
By choosing $C=\frac{1}{|Q|}$ and $c=\frac{1}{4kd}$, and allowing $M$ to grow without limit,
we get the claim of the theorem.
\end{proof}

Because SAT is NP-complete and has a superpolynomial lower bound on its
deterministic time complexity by Theorem \eqref{SAT_complexity}, we get
\begin{corollary} $NP > P$.
\end{corollary}

{\bf Acknowledgements}. The author is very grateful to Lauri Hella for many extremely
important critical remarks and fruitful discussions, in the course of an almost
non-denumerable number of revisions to the current manuscript.

\end{document}